\documentclass[12pt]{article}
\usepackage{amsmath,amssymb,graphicx,amsthm}

\newtheorem{theorem}{Theorem}[section]
\newtheorem*{theorem:repeat}{\tref{butterflystab}}
\newtheorem*{theorem:repeatmain}{\tref{main}}
\newtheorem{lemma}[theorem]{Lemma}
\newtheorem{corollary}[theorem]{Corollary}

\newtheorem{prop}[theorem]{Proposition}
\newtheorem{conjecture}[theorem]{Conjecture}

\newtheorem{definition}[theorem]{Definition}
\newtheorem{construction}[theorem]{Construction}

\newcommand\lref[1]{Lemma~\ref{lem:#1}}
\newcommand\tref[1]{Theorem~\ref{thm:#1}}
\newcommand\cref[1]{Corollary~\ref{cor:#1}}

\newcommand\sref[1]{Section~\ref{sec:#1}}
\newcommand\pref[1]{Proposition~\ref{prop:#1}}

\newcommand\cF{{\mathcal F}}

\newcommand\cI{{\mathcal I}}

\newcommand\cM{{\mathcal M}}

\title{On the number of maximal intersecting $k$-uniform families  and further applications of Tuza's set pair method}

\author{Zolt\'an L\'or\'ant Nagy\\
 \small MTA--ELTE Geometric and Algebraic Combinatorics Research Group \\[-0.8ex]
\small  H--1117 Budapest, P\'azm\'any P.\ s\'et\'any 1/C, Hungary.\\[-0.8ex]
\small \tt  nagyzoli@cs.elte.hu
\and 
Bal\'azs Patk\'os\thanks{J\'anos Bolyai Research Scholarship of the Hungarian Academy of Sciences}\\
\small MTA--ELTE Geometric and Algebraic Combinatorics Research Group\\[-0.8ex]
\small  H--1117 Budapest, P\'azm\'any P.\ s\'et\'any 1/C, Hungary and\\[-0.8ex]
\small  Alfr\'ed R\'enyi Institute of Mathematics, Hungarian Academy of Sciences\\[-0.8ex]
\small\tt patkosb@cs.elte.hu, patkos@renyi.hu}

\date{
}

\begin{document}
\maketitle

\begin{abstract}

We study the function $M(n,k)$ which  denotes the number of maximal $k$-uniform intersecting families $\cF\subseteq \binom{[n]}{k}$.  
Improving a bound of Balogh, Das, Delcourt, Liu and Sharifzadeh   on $M(n,k)$,   we determine the order of magnitude of $\log M(n,k)$ by proving that for any fixed $k$,  $M(n,k) =n^{\Theta(\binom{2k}{k})}$ holds. Our proof is based on Tuza's set pair approach.

The main idea is to bound
the size of the largest possible point set of a  cross-intersecting system. We also introduce and investigate some related functions and parameters. 
\end{abstract}

\section{Introduction}
Many problems in extremal combinatorics ask for the maximum possible size that a combinatorial structure can have provided it satisfies some prescribed property $P$.
 Questions  about the size of the `underlying set' of the combinatorial structure are much less frequently asked. (In many cases, this size is part of property $P$.) This note is devoted to an application of  Tuza's set pair method \cite{T} 
 which  provides good bounds for problems of the first type through results on problems of the second type.\\ The starting point of Tuza's method is the following celebrated theorem of Bollob\'as.

\begin{theorem}[Bollob\'as, \cite{B}]
\label{thm:boll} Let $A_1,A_2,\dots, A_m$ and $B_1,B_2,\dots,B_m$ be sets such that $|A_i|\le k$ and $|B_i|\le l$ hold for all $1\le i \le m$. Let furthermore these sets satisfy
\begin{enumerate}
\item[(1)]
$A_i \cap B_i=\emptyset$  for all $1\le i\le m$,
\item[(2)]
$A_i\cap B_j\neq \emptyset$  for all $1\le i,j \le m$, $i\neq j$.
\end{enumerate}
Then $\sum_{i=1}^m\frac{1}{\binom{|A_i|+|B_i|}{|A_i|}}\le 1$, in particular $m\le \binom{k+l}{l}$ holds.
\end{theorem}

Pairs satisfying the conditions of \tref{boll} will be called \textit{cross intersecting set pairs} and if we want to emphasize the size condition of the $A_i$'s and $B_j$'s, then we call the system \textit{$(k,l)$-cross intersecting}.\\ Modifying Lov\'asz's proof \cite{L} of \tref{boll}, Frankl \cite{F} and later Kalai \cite{K} obtained the following skew version of the result.

\begin{theorem}[Frankl]
\label{thm:skewboll} Let $A_1,A_2,\dots, A_m$ and $B_1,B_2,\dots,B_m$ be sets  such that $|A_i|\le k$ and $|B_i|\le l$, satisfying the conditions 
\begin{enumerate}
\item[(1)]
$A_i \cap B_i=\emptyset$  for all $1\le i\le m$,
\item[(2')]
 $A_i\cap B_j\neq \emptyset$  for all $1\le i<j \le m$.
\end{enumerate}
Still the bound $m\le \binom{k+l}{l}$ remains valid.
\end{theorem}
Pairs satisfying the conditions of \tref{skewboll} will be called \textit{skew cross intersecting set pairs}.

The vertex set of a (skew) cross intersecting system of set pairs is $V=\bigcup_{i=1}^m(A_i\cup B_i)$. Tuza was interested in the maximum possible size of the vertex set of a $(k,l)$-cross intersecting system. Let us write
\[
n(k,l)=\max\left\{\left|\bigcup_{i=1}^m(A_i\cup B_i)\right|: (A_1,B_1),\dots,(A_m,B_m) \text{ is a $(k,l)$-cross intersecting system} \right\}.
\]
Obviously, by \tref{boll}, we have $n(k,l)\le (k+l)\binom{k+l}{l}$, but the following upper bound was obtained in \cite{T}.

\begin{theorem}[Tuza \cite{T}]
\label{thm:tuza}
For positive integers $k\le l$ we have 
\[
\frac{1}{4}\binom{k+l+1}{k+1}< n(k,l)\le \sum_{i=1}^{2k-2}\binom{i}{\lfloor i/2\rfloor}+\sum_{i=2k-1}^{k+l-1}\binom{i}{l}<\binom{k+l+1}{k+1}.
\]
\end{theorem}

\sref{max} is devoted  to prove another application of the set pair method, the main result of this note.
Apart from antichains the most studied set families are \textit{intersecting} families. We say that $\cF\subseteq 2^{[n]}$ is intersecting if $F_1 \cap F_2 \neq \emptyset$ holds for all $F_1,F_2 \in \cF$. It is well-known that all \textit{maximal} (unextendable) intersecting families $\cF\subseteq 2^{[n]}$ have size $2^{n-1}$. 
(Here and thereafter $[n]$ stands for the set $\{1,2,\ldots, n\}$.)
The investigation of $\lambda(n)$ and $\Lambda(n)$, the number of intersecting and maximal intersecting families, respectively, was started in \cite{E}. The exact values are known for small $n$ \cite{BMMV} and determining the order of magnitude of $\log \lambda(n)$ and $\log\Lambda(n)$ is an easy exercise.

 Recently, Balogh, Das, Delcourt, Liu, and Sharifzadeh \cite{BDDLS} studied the uniform version of the problem. The famous Erd\H os-Ko-Rado theorem \cite{EKR} states that an intersecting family $\cF\subseteq \binom{[n]}{k}$ can have size at most $\binom{n-1}{k-1}$ if $2k\le n$ holds. Furthermore, intersecting families achieving the extremal size consist of all $k$-sets containing a fixed element of $[n]$ provided $2k<n$. Balogh et al. define the  function $N(k)$ with the property that if $n\ge N(k)$, then the number of $k$-uniform intersecting families is $2^{(1+o(1))\binom{n-1}{k-1}}$. In their proof they obtain an upper bound on the number $M(n,k)$ of maximal $k$-uniform intersecting families. Here we improve on this bound and  we determine the order of magnitude of the exponent of $n$ in $M(n,k)$ for any fixed $k$.

\begin{theorem}
\label{thm:main} For any fixed integer $k$, as $n$ tends to infinity the function $M(n,k)$ satisfies
\[
 M(n,k)=n^{\Theta(\binom{2k}{k})}.
\]
Moreover, $$\frac{1}{8}\le \limsup_n\frac{\log M(n,k)}{\binom{2k}{k}\log n}\le 1.1 \mbox{ \ \ and \ \ }  \limsup_k \limsup_n\frac{\log M(n,k)}{\binom{2k}{k}\log n}\le 1$$ holds.
\end{theorem}

The proof of \tref{main} uses the upper bound in \tref{tuza}. In \sref{vertices} we first prove an upper bound on $n(k,l)$ that is weaker than that of \tref{tuza}, but its proof technique is completely different: it involves skew cross intersecting systems. Therefore it is natural to introduce the following analog of the function $n(k,l)$: 
\[
n_1(k,l)=\max\left\{\left|\bigcup_{i=1}^m(A_i\cup B_i)\right|: (A_1,B_1),\dots,(A_m,B_m) \text{ is a $(k,l)$-skew cross intersecting system} \right\}.
\]
We finish Section 3 by presenting lower and upper bounds on $n_1(k,l)$.

\vskip 0.3truecm

Before starting to prove our theorems let us mention that there has been recent activity \cite{BNR,GHO,HK1,HK2} on the following problem of Balogh, Bohman and Mubay \cite{BBM} related to maximal intersecting families: let $H(n,k,p)$ denote the random $k$-uniform hypergraph obtained from $\binom{[n]}{k}$ by keeping each edge with probability $p$ independently of all other choices. What is the size of the largest intersecting family in $H(n,k,p)$ and what its structure looks like. The above mentioned papers settled this question for all interesting values of $k=k(n)$ and $p=p(n)$.


\section{Proof of the main theorem}
\label{sec:max}

We start with the lower bound of \tref{main}. For a family $\cF$ of sets its \textit{covering number} $\tau(\cF)$ is the minimum size that a \textit{transversal} $G$ of $\cF$ can have. A transversal of $\cF$ is a set meeting all $F\in \cF$. Clearly, $\tau(\cF)\le k$ holds for all intersecting $k$-uniform families as any set in $\cF$ is a transversal. Let us define the function $f(k)$ by
\[
f(k)=\max\{|\cup_{F\in \cF}F|: \cF ~\text{is $k$-uniform intersecting with}\ \tau(\cF)=k\}.
\]

Note that $f(k)$ is finite (see  \cite{EL} ), while the condition $\tau(\cF)=k$ is essential in the sense that $|\cup_{F\in \cF}F|$ could be arbitrarily large if $\cF$ was $k$ uniform intersecting with $\tau(\cF)<k$.
Many similar functions concerning $k$-uniform intersecting families with covering number $k$ were introduced and studied in \cite{EL} (and later by many other researchers). The following example is due to Tuza \cite{T}.

\begin{construction}
 Let $|Y| =2k-2$. For each partition $Y$ as  $E\cup E'= Y$, $|E| = |E'| =k-1$ we take a new points $x$, and set $E \cup \{x\}$, $E' \cup \{x\}$. In this way we obtain $\binom{2k-2}{k-1}$ $k$-element sets forming an intersecting family with covering number $k$, such that the union of these sets consists of $2k-2+\frac{1}{2}\binom{2k-2}{k-1}$ points. 
\end{construction}

 \begin{corollary} $\frac{1}{8}\binom{2k}{k} <2k-2+\frac{1}{2}\binom{2k-2}{k-1} \leq f(k)$.
\end{corollary}
\vskip 0.3truecm

The following proposition finishes the proof of the lower bound of \tref{main}.

\begin{prop}
For any positive integers $k$ and $n$ we have $\binom{n}{f(k)}\le M(n,k)$.
\end{prop}

\begin{proof}
Consider a $k$-uniform intersecting family $\cF$ with $\tau(\cF)=k$ and $|\cup_{F\in\cF}F|=f(k)$. As adding more sets to $\cF$ can only increase the size of the union, we may assume that $\cF$ is maximal intersecting. Every set $X\in \binom{[n]}{f(k)}$ contains at least one family $\cF_X$ isomorphic to $\cF$. As $\cF_X\neq \cF_Y$ whenever $\cup_{F\in\cF_X}F=X\neq Y=\cup_{F\in\cF_Y}F$, we have at least $\binom{n}{f(k)}$ different maximal intersecting $k$-uniform subfamilies of $\binom{[n]}{k}$.
\end{proof}

As we mentioned in the proof, the value of $f(k)$ is attained at a maximal intersecting family. Note that such a family is unextendable not only by any $k$-subsets of its underlying set, but by any $k$-sets in the universe at all. This kind of maximal intersecting set systems were studied a lot, the best known upper bound on $f(k)$ is due to Majumder \cite{M}, stating that $f(k) \leq (1+o(1))\frac{3}{2}\binom{2k-2}{k-1}$.

\vskip 0.2truecm

We now turn our attention to the upper bound of \tref{main}. We start by describing the basic ideas of Balogh, Das, Delcourt, Liu, and Sharifzadeh \cite{BDDLS}. For a family $\cF\subseteq \binom{[n]}{k}$ of sets let $\cI(\cF)=\{G\in\binom{[n]}{k}: \forall F\in\cF: F\cap G\neq \emptyset\}$, that is if $\cF$ is intersecting, then $\cI(\cF)$ denotes the family of those sets that can be added to $\cF$ preserving the intersecting property. Clearly, $\cF$ is maximal intersecting if and only if $\cF\subseteq \cI(\cF)$ holds with equality. For any maximal intersecting family we can assign a subfamily $\cF_0\subseteq \cF$ that is minimal with respect to the property $\cI(\cF_0)=\cF$ (note that $\cF_0$ is not necessarily unique). Then by definition, for every $F\in\cF_0$ there exists a $G\in \cI(\cF_0\setminus \{F\})\setminus \cF$, thus this $G$ intersects all sets in $\cF$ but $F$. Therefore the sets of $\cF_0$ and their pairs $G$ satisfy the condition of \tref{boll} and thus by  above, we obtain that $|\cF_0|\le \binom{2k}{k}$. Moreover, if $\cF_0=\{F_1,F_2,\dots, F_s\}$ and $G_i$ is a set in $\cI(\cF_0\setminus \{F_i\})\setminus \cF$, then the set of pairs $\{(A_i,B_i)\}^{2s}_{i=1}$ with
$A_i=F_i, B_i=G_i$ for $1\le i\le s$ and $A_i=G_{i-s}, B_i=F_{i-s}$ for $s<i\le 2s$ is skew cross intersecting and thus by \tref{skewboll} the inequality $|\cF_0|\le \frac{1}{2}\binom{2k}{k}$ holds. Since the mapping of maximal intersecting families via $\cF \mapsto \cF_0$ is injective, Balogh et al. obtained $M(n,k)\le \sum_{j=1}^{\frac{1}{2}\binom{2k}{k}}\binom{\binom{n}{k}}{j}=O(n^{\frac{k}{2}\binom{2k}{k}})$. Comparing this to our lower bound, we see that the exponent is off only by a factor of $4k$. In what follows we show how to improve the previous upper bound.

\vskip 0.3truecm

In order to obtain our upper bound, we will use the function $n(k,l)$. As the argument of Balogh et al. yields a cross intersecting system in which sets of the first co-ordinate form an intersecting family on their own, we introduce the following:
\[
g(k)=\max\{|\cup_{i=1}^sA_i|: \{(A_i,B_i)\}^s_{i=1} ~\text{is $(k,k)$-cross intersecting and $\{A_i\}^s_{i=1}$ is intersecting} \}.
\]
By definition, we have $g(k)\le n(k,k)$. The following lemma and proposition complete the proof of the upper bound of \tref{main}.

\begin{lemma}
\label{lem:bound} $M(n, k)\le 2^{2^{g(k)}}\binom{n}{g(k)}$.
\end{lemma}

\begin{proof}
Let us consider a function $f$ that maps any maximal intersecting $k$-uniform family $\cF$ to one of its subfamily $\cF_0$ that is minimal with respect to the property that $\cI(\cF_0)=\cF$. As mentioned earlier, $f$ is injective, $\cF_0$ is intersecting and the set of pairs $\{(F_i,G_i)\}_{i=1}^{|\cF_0|}$ is $(k,k)$-cross intersecting. Thus by definition $|\cup_{F\in\cF_0}F|\le g(k)$ holds. Therefore the set families that can be the image of a maximal intersecting $k$-uniform family with respect to $f$ are subfamilies of $2^X$ for some $X\in \binom{[n]}{g(k)}$. The number of such families is not more than $2^{2^{g(k)}}\binom{n}{g(k)}$.
\end{proof}

Though it was not mentioned in \cite{T}, the summation form of the upper bound of \tref{tuza} provides much better estimation in the case $k=l$.

\begin{prop}\label{pontos} Let $S(k)$ denote Tuza's upper bound on $n(k,k)$ in \tref{tuza}, that is, $S(k) = \sum_{i=1}^{2k-1}\binom{i}{\lfloor i/2\rfloor}$. Then
\begin{enumerate}
\item[(i)]
$g(k)\leq n(k,k)\leq S(k) \leq 1.1\cdot\binom{2k}{k}$,
\item[(ii)]
$s(k):=\frac{S(k)}{\binom{2k}{k}} \rightarrow 1$ if $k \rightarrow \infty.$
\end{enumerate}
\end{prop}

\begin{proof} Statement (i) can be confirmed easily for $k\leq 4$, and for $k>4$ simple inductive argument works.  For statement (ii), one can easily check that $s(k)>1$ holds, and the sequence $s(k)$ is monotone decreasing from $k=4$. Moreover the limit cannot be greater than $1$, since  if $s(k)>(1+\varepsilon)$  held with a fixed $\varepsilon>0$ for all $k$, that would imply $\frac{s(k+1)}{s(k)}\leq \frac{4k+3}{4k+2}\frac{1}{(1+\varepsilon)}$, a contradiction.
\end{proof}


\section{Bounds on the size of the vertex set}
\label{sec:vertices}

In the forthcoming section we present lower and upper bounds on $n(k,l)$ and $n_1(k,l)$, that is, on the maximal size of the underlying set of a (skew) cross intersecting system. 

\begin{construction}[Erd\H os-Lov\'asz, \cite{EL}] Let $Y$ be a set of $2k-2$ elements. For each subset $A'\subset Y$, $|A'|=k-1$, we assign a set pair $(A, B)$ such that $|A|=k=|B|$ holds,  $A'\subset A$, $(Y\setminus A')\subset B$ and the one element sets $A\setminus Y$, $B\setminus Y$ are disjoint.   In this way we obtain $\binom{2k-2}{k-1}$  set pairs  such that the union of these sets consists of $2k-2+2\binom{2k-2}{k-1}$ points. 
\end{construction}

This construction slightly improves  the general lower bound of  \tref{tuza} on $n(k,l)$ in the special case $k=l$. Thus in view of Proposition \ref{pontos}, this provides

\begin{prop}  $2k-2+2\binom{2k-2}{k-1} \leq n(k,k) \leq 1.1\cdot\binom{2k}{k}.$
\end{prop}
\medskip

\noindent In the spirit of Tuza's approach, the following upper bound  is obtained on $n(k,l)$. 

\begin{lemma}
\label{lem:g}
$n(k,l)\le \binom{k+l}{l+1}+\binom{k+l}{k+1}$.
\end{lemma}

\begin{proof}
Let $\{(A_i,B_i)\}^s_{i=1}$ be a set of cross intersecting pairs with $|A_i|\le k$ and $|B_i|\le l$ for all $1\le i\le s$. Let $\alpha_t=|\{i: |A_i\setminus (\cup_{j=1}^{i-1}(A_j\cup B_j))|\ge t\}|$ and $\beta_t=|\{i: |B_i\setminus (\cup_{j=1}^{i-1}(A_j\cup B_j))|\ge t\}|$. Clearly, we have 
\[
 |\bigcup_{i=1}^s(A_i\cup B_i)|=\sum_{t=1}^k{\alpha_t}+   \sum_{t=1}^l{ \beta_t} .
\]
Let us bound $\beta_t$. Observe that if we write $B'_i=B_i\cap (\cup_{j=1}^{i-1}(A_j\cup B_j))$, then the set of pairs $\{(A_i,B'_i)\}_{i=1}^s$ is skew cross intersecting. Moreover  $$|B_i\setminus (\cup_{j=1}^{i-1}(A_j\cup B_j))|\ge t$$ holds for $i$ if and only if $|B'_i|\le l-t$. Hence   $\beta_t$ is equal to the number of skew cross intersecting set pairs $\{(A_i,B'_i)\}$ where $|A_i|\le k$ and $|B'_i|\le l-t$.
 Applying \tref{skewboll} we obtain $\beta_t\le \binom{k+l-t}{k}$, and as the role of $\alpha_t$ and $\beta_t$ is similar we also have $\alpha_t\le \binom{k+l-t}{l}$. Consequently,
\[
|\bigcup_{i=1}^s(A_i\cup B_i)|\le \sum_{t=1}^k\binom{k+l-t}{l}+\sum_{t=1}^l\binom{k+l-t}{k}= \binom{k+l}{l+1}+\binom{k+l}{k+1}.
\]
\end{proof}

 This slightly improves the bound $\binom{k+l+1}{l+1}$ of \tref{tuza} when $k=l$. However, in view of  Proposition \ref{pontos}, Tuza's bound involving a summation is still better even in this case.

Recall that
\[
n_1(k,l)=\max\left\{\left|\bigcup_{i=1}^m(A_i\cup B_i)\right|: (A_1,B_1),\dots,(A_m,B_m) \text{ is a $(k,l)$-skew cross intersecting system} \right\}.
\]

\noindent Our second result gives lower and upper bounds on $n_1(k,l)$. In order to do this, we recall what a reverse lexicographic order (or sometimes called \textit{colex order}) is.

\begin{definition} A reverse lexicographic order of the $k$-element subsets of $[n]$  is defined by the relation\\ \centerline{ $C <D$ for $C, D\in \binom{[n]}{k} \Leftrightarrow $ \ the largest element of the symmetric difference $C\triangle D$ is in $D$.}
\end{definition}

\begin{construction}\label{skewpelda} Let $Y$ be the set $Y=[k+l]$. Consider the reverse lexicographic order of all the $k$-element subsets of $Y$. Let $A_i=\{a_{i,1}, a_{i,2}, \ldots, a_{i,k}\}$ $(i=1\ldots\binom{k+l}{k})$  be the $i$th set in this order with the $a_{i,j}$'s enumerated in increasing order, and let $B_i$ be defined as follows. $B_i\cap Y= [a_{i,k}] \setminus A_i$ and let all the sets $B_i\setminus Y$ be pairwise disjoint for all $i$ such that $|B_i|=k$.
\end{construction}

 \begin{prop} $k+l+\binom{k+l}{k+1}\leq n_1(k,l)$.
\end{prop}

\begin{proof} Construction \ref{skewpelda} provides  a $(k,l)$-skew cross intersecting set system.
Indeed, $A_i\cap B_i=\emptyset$, while  $A_i\cap B_j\neq\emptyset$ for $i<j$, since $A_i\subseteq [a_{i,k}]\subseteq [a_{j,k}]$, hence $A_i\cap B_j\supseteq A_i\cap ([a_{j,k}]\setminus A_j)\neq \emptyset$. Observing that the number of $k$-sets $A_j$ with $a_{j,k}=k+c$ is $\binom{k+c-1}{k-1}$  \ (assuming $c\leq l$), we get  $$\left|\bigcup_{i=1}^m(A_i\cup B_i)\right|= k+l+ \sum_{c=0}^l  (l-c)\binom{k+c-1}{k-1}.$$ Next notice that $$\sum_{c=0}^l  (l-c)\binom{k+c-1}{k-1}=\sum_{x=0}^{l-1}\sum_{c=0}^x  \binom{k+c-1}{k-1} = \sum_{x=0}^{l-1} \binom{k+x}{k}=\binom{k+l}{k+1},$$ hence the result follows.
\end{proof}

Note that Construction \ref{skewpelda} shows that the calculation in \lref{g} to bound $\beta_t$ is tight and thus to obtain better bounds on $n_1(k,l)$ one has to use further ideas.

\vskip 0.2truecm

The proof below of the upper bound on $n_1(k,l)$ is based on Tuza's approach \cite{T} to determine $n(k,l)$.

\begin{prop}
\label{prop:skewpoints} Let $k\leq l$ be positive integers. Then $n_1(k,l)\le  \binom{k+l+2}{k+1} -\binom{k+l}{k}-2$ holds.
\end{prop}

\begin{proof}
Let $\{(A_i,B_i)\}^s_{i=1}$ be a set of skew cross intersecting pairs with $|A_i|\le k$ and $|B_i|\le l$ for all $1\le i\le s$ and let us define $S_0=[s]$ and $\cM_0=\{(A^0_i,B^0_i)\}^s_{i=1}$ with $A^0_i=A_i$ and $B^0_i=B_i$. If $S_{j}$ and $\cM_j$ are defined for some $j\le k+l-2$, then let $S_{j+1}\subset S_j$ be an index set minimal with respect to the property that
\[
\bigcup_{i\in S_j}(A^j_i\cup B^j_i)=\bigcup_{i\in S_{j+1}}(A^j_i\cup B^j_i).
\]
By minimality for every $i\in S_{j+1}$ there exists a point $x_i \in (A^j_i\cup B^j_i)\setminus \bigcup_{l \in S_{j+1}\setminus \{i\}} (A^j_l\cup B^j_l)$. Let us define $A^{j+1}_i=A^j_i\setminus \{x_i\}, B^{j+1}_i=B^j_i\setminus \{x_i\}$ for all $i\in S_{j+1}$ and set $\cM_{j+1}=\{(A^{j+1}_i,B^{j+1}_i): i\in S_{j+1}\}$. Observe that $\cM_j$ is skew intersecting for all $1\le j\le k+l-1$ with $|A^j_i \cup B^j_i|\le k+l-j$ for all $i \in S_j$ and furthermore 
\[
\bigcup_{i=1}^s(A_i \cup B_i)=\sum_{j=1}^{k+l-1}|\cM_j|.
\]
In Tuza's original proof the $\cM_j$'s are cross intersecting and therefore he can use Bollob\'as's inequality to obtain $|\cM_j|\le \binom{k+l-j}{\lceil \frac{k+l-j}{2}\rceil}$ for any $j$ and $|\cM_j|\le \binom{k+l-j}{k}$ if $j\le l-k$. As Bollob\'as's inequality is not valid for skew intersecting pairs, therefore we partition $\cM_j$ into some subsystems indexed by the pairs $(|A_i\setminus A^j_i|, |B_i\setminus B^j_i|)$. Note that by the construction of the $\cM_j$'s for the index pairs $(a,b)$ we have $0\le a,b\le j$, $a+b=j$, $a\leq k$ and $b\leq l$. For such a subsystem of $\cM_j$,  indexed by $(a,b)$,   we can apply \tref{skewboll} and obtain the upper bound $\binom{k+l-j}{k-a}$. Thus adding these up for all $\cM_j$, $j\in [1, k+l-1]$, we get 

$$\sum_{j=1}^{k+l-1} \sum_{a=0}^j \binom{k+l-j}{k-a} = \sum_{\substack{\beta \leq k \\  \alpha \leq \beta+l}} \binom{\alpha}{\beta} - \binom{0}{0}-\binom{k+l}{k}.$$

Here $$\sum_{\substack{\beta \leq k \\  \alpha \leq \beta+l}} \binom{\alpha}{\beta} = \sum_{\beta=0}^k \sum_{\gamma=0}^l\binom{\beta+\gamma}{\beta} = \sum_{\beta=0}^k \binom{\beta+l+1}{\beta+1}= \binom{k+l+2}{k+1} -1,$$ 
confirming the statement.
\end{proof}

In \cite{T3}, Tuza proposed the investigation of the so-called \textit{weakly cross intersecting set pair systems}, which are closely related to the cross intersecting set pair systems.

\begin{definition}
Let $A_1,A_2,\dots, A_m$ and $B_1,B_2,\dots,B_m$ be sets such that $|A_i|= k$ and $|B_i|= l$ holds for all $1\le i \le m$. Let furthermore these sets satisfy
\begin{enumerate}
\item[(1)]
$A_i \cap B_i=\emptyset$  for all $1\le i\le m$,
\item[(2)]
$A_i\cap B_j\neq \emptyset$  or $A_j\cap B_i\neq \emptyset$   for all $1\le i,j \le m$, $i\neq j$.
\end{enumerate}
Then the system $\{(A_i,B_i)\}^m_{i=1}$ is called a $(k,l)$-weakly cross intersecting set pair system.\\
Let $m_{max}(k,l)$ denote the largest $m\in \mathbb{Z}$ for which a  $(k,l)$-weakly cross intersecting set pair system $\{(A_i,B_i)\}^m_{i=1}$ exists.
\end{definition}

Surprisingly, much less is known about the maximum size of a weakly cross intersecting set pair system compared to the original case. Concerning the upper bound, Tuza showed \cite{T3} that $m_{max}(k,l)\leq \frac{(k+l)^{k+l}}{k^kl^l}$.  Kir\'aly, Nagy, P\'alv\"olgyi and Visontai gave a construction \cite{kiraly} that provides $\lim \inf_{k+l\rightarrow \infty} m_{max}(k,l)\geq (2-o(1))\binom{k+l}{k}$. Moreover, they conjectured the latter result to be sharp:

\begin{conjecture}[\cite{kiraly}]\label{weak}
$$m_{max}(k,l)\leq 2\binom{k+l}{k}.$$
\end{conjecture}

These questions motivate the investigation of

\[
n_2(k,l)=\max\left\{\left|\bigcup_{i=1}^m(A_i\cup B_i)\right|: (A_1,B_1),\dots,(A_m,B_m) \text{ is a (k,l)-weakly cross intersecting system} \right\}.
\]

First, observe that the idea of the proof of \pref{skewpoints} works smoothly to obtain an upper bound on $n_2(k,l)$, since we may define weakly cross intersecting set pair systems $\cM_j$ similarly from a given $(k,l)$-weakly cross intersecting set pair system. Thus the exact upper bound only depends on $m_{max}(k,l)$. Hence, assuming that Conjecture \ref{weak} holds, we get the double of the upper bound of \pref{skewpoints}.

A lower bound follows from 

\begin{construction}\label{harmas} Let $Y$ be a set of $k+l-1$ elements with  $k \le l$. Assign a subset $B_i'\subset (Y\setminus A_i')$ of size $l-1$ to each  $k-1$ element subset $A_i'\subset Y$ in such a way that the sets $B_i'$ are distinct. This can be done due to the K\H onig-Hall theorem and the fact that $k \le l$. For each $A_i'$, assign furthermore three distinct elements $x_i,y_i,z_i\not\in Y$. Take the set pairs $(A'_i\cup \{x_i\}, B'_i\cup \{y_i\})$, $(A'_i\cup \{y_i\}, B'_i\cup \{z_i\})$, $(A'_i\cup \{z_i\}, B'_i\cup \{x_i\})$ for all $i$. 
  This way we obtain $3\binom{k+l-1}{k-1}$  set pairs  such that the union of these sets consists of $k+l-1+3\binom{k+l-1}{k-1}$ points. 
\end{construction}

\begin{prop}
$k+l-1+3\binom{k+l-1}{k-1}\leq n_2(k,l).$ 
\end{prop}

\begin{proof}
The proposition follows from the fact that  Construction \ref{harmas} provides a weakly cross intersecting set pair system, which is easy to see.
\end{proof}

\noindent\textbf{Acknowledgment.} We would like to thank an anonymous referee for pointing out an error in a previous version of the manuscript and for their many helpful remarks to improve the presentation of the paper.

\end{document}